\documentclass[12pt]{article}

\usepackage[margin = 1in]{geometry}
\usepackage{setspace}

\usepackage[T1]{fontenc}
\usepackage{charter} 
\usepackage{euler} 
\usepackage{float}
\usepackage{soul} 
\usepackage{bbm} 
\usepackage{verbatim}
\usepackage{amsmath}
\usepackage{amsthm}
\usepackage{amssymb}
\usepackage{textcomp}
\usepackage{graphicx}
\usepackage[english]{babel}
\usepackage{tikz}
\usepackage[linguistics]{forest}
\usepackage{dblfloatfix}
\usepackage{csquotes}
\usepackage{hyperref}
\usepackage{nameref}
\usepackage{theoremref}
\usepackage{natbib}
\usepackage{comment}

\usepackage{enumerate}

\allowdisplaybreaks 

\usepackage{algorithm}
\usepackage{algpseudocode}

\newtheorem{theorem}{Theorem}[section]

\newtheorem{lemma}[theorem]{Lemma}

\newtheorem*{remark}{Remark}
\newtheorem*{theorem*}{Theorem}

\usepackage{titlesec}
\titleformat{\section}{\sc\Large}{\thesection}{1em}{}
\titleformat{\subsection}{\sc\large}{\thesubsection}{1em}{}

\usepackage{tcolorbox}{}
\tcbuselibrary{theorems}

\newtcbtheorem[number within=section]{Theorem}{Theorem}{
  colback=white!75!gray,
  colframe=gray!75!black,
  fonttitle=\bfseries,
}{thm}
\newcounter{Theorem}
\numberwithin{Theorem}{section} 

\newtcbtheorem{Remark}{Remark}{
  colback=white!75!gray,
  colframe=gray!75!black,
  fonttitle=\bfseries,
}{rem}

\newtcbtheorem{Algorithm}{Algorithm}{
  colback=white!75!gray,
  colframe=gray!75!black,
  fonttitle=\bfseries,
}{rem}

\newcommand{\1}{\mathbf{1}} 
\newcommand{\prob}{\mathbf{P}} 
\newcommand{\ex}{\mathbf{E}} \newcommand{\EXP}{\mathbf{E}} 

\newcommand{\R}{\mathbbm{R}} 
\newcommand{\bin}{\ensuremath{\operatorname{Binomial}}} 
\newcommand{\ber}{\ensuremath{\operatorname{Bernoulli}}} 
\newcommand{\unif}{\ensuremath{\operatorname{Uniform}}} 
\newcommand{\exponential}{\ensuremath{\operatorname{Exponential}}} 
\newcommand{\poisson}{\ensuremath{\operatorname{Poisson}}} 
\newcommand{\convdist}{\ensuremath{\xrightarrow[]{\mathcal{L}}}} 
\newcommand{\convprob}{\ensuremath{\xrightarrow[]{\mathcal{\mathbb{P}}}}} 
\newcommand{\var}{\ensuremath{\operatorname{Var}}} 
\newcommand{\dist}{\ensuremath{\overset{\mathcal{L}}{=}}} 

\newcommand{\cE}{\mathcal{E}}

\title{\bf{Uniform temporal trees}
  \thanks{Luc Devroye acknowledges the support of NSERC grant A3456.
    G\'abor Lugosi acknowledges the support of Ayudas Fundación BBVA a
Proyectos de Investigación Científica 2021 and
the Spanish Ministry of Economy and Competitiveness grant PID2022-138268NB-I00, financed by MCIN/AEI/10.13039/501100011033,
FSE+MTM2015-67304-P, and FEDER, EU.).}
}

\author{
\bf{Caelan Atamanchuk} \\
Department of Mathematics and Statistics \\
McGill University \\
caelan.atamanchuk@gmail.com
\and
\bf{Luc Devroye} \\
School of Computer Science \\
McGill University, \\
lucdevroye@gmail.com
\and
\bf{G\'{a}bor Lugosi} \\
Department of Economics and Business, \\
Pompeu  Fabra University, Barcelona, Spain \\
ICREA, Pg. Lluís Companys 23, 08010 Barcelona, Spain \\
Barcelona Graduate School of Economics \\
gabor.lugosi@gmail.com
}

\begin{document}

\setstretch{1}

\maketitle

\begin{abstract}
Motivated by the study of random temporal networks, we introduce a class of random trees that we coin \emph{uniform temporal trees}.
A uniform temporal tree is obtained by assigning independent uniform $[0,1]$ labels to the edges of a rooted complete infinite  $n$-ary tree and keeping only those vertices for which the path from the root to the vertex has decreasing edge labels.
 The $p$-percolated uniform temporal tree, denoted by $\mathcal{T}_{n,p}$, is obtained 
similarly, with the additional constraint that the edge labels on
each path are all below $p$. 
We study several properties of these trees, including their size, height, the typical depth of a vertex, and degree distribution. In particular, we establish a limit law for the size of $\mathcal{T}_{n,p}$ which states that $\frac{|\mathcal{T}_{n,p}|}{e^{np}}$ converges in distribution to an $\exponential(1)$ random variable as $n \to \infty$. For the height $H_{n,p}$, we prove that $\frac{H_{n,p}}{np}$ converges  to $e$ in probability. Uniform temporal trees show some remarkable similarities to uniform random recursive trees. 
\end{abstract}

\setstretch{1.3}

\tableofcontents

\begin{figure}[H]
    \centering
    \includegraphics[scale=0.5]{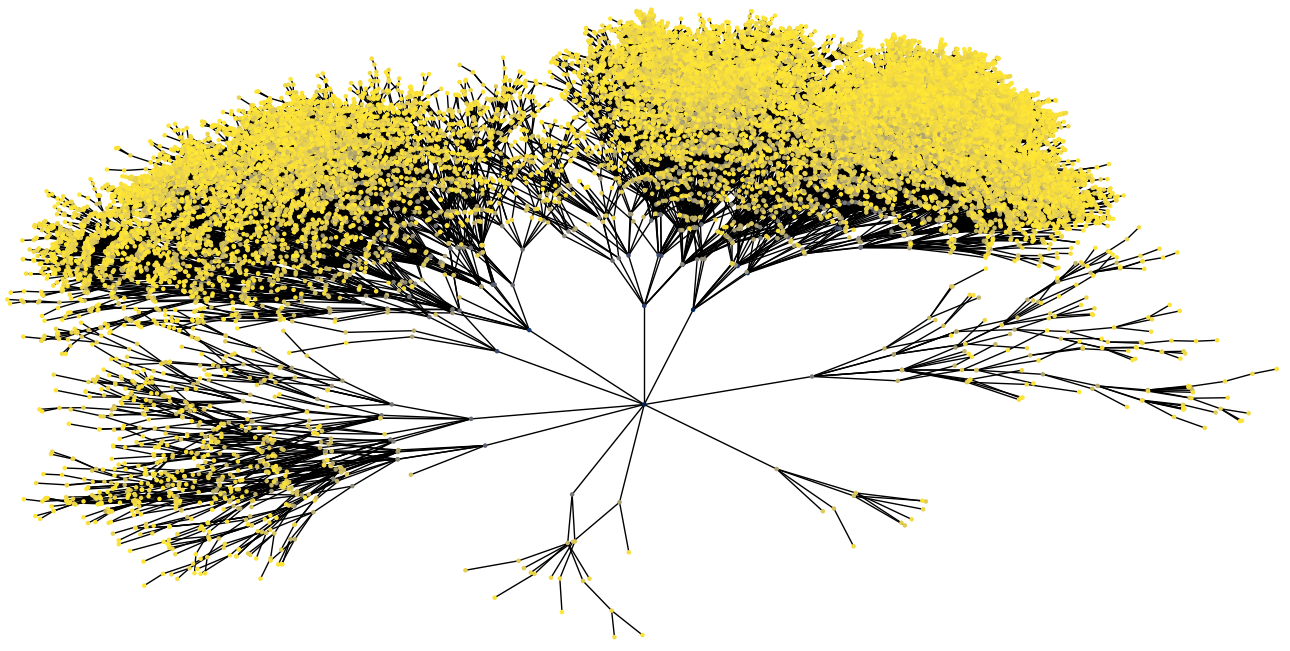}
    \caption{A uniform temporal tree with $n=10$.}
    \label{fig:pic}
\end{figure}

\newpage

\section{Introduction}

In network science, the graph modeling the network is often equipped with edge labels representing a time stamp. For example, in a network describing human interactions, the network's vertices represent individuals, edges stand for encounters, and the edges may be labeled by the time the encounter happens. Such \emph{temporal} networks allow one to study the spread of an infection or information (see \cite{holme2012}, \cite{holme2013}, \cite{holme2019}, \cite{hosseinzadeh2022}, \cite{kumar2020}).

A simple mathematical model for temporal networks that has been gaining attention
is \emph{random temporal graphs}. In this model, the time stamps are obtained by assigning a uniform random permutation to the edge set. If one is only interested in the ordering of the edge labels, equivalently, every edge of a graph is assigned an independent random label, uniformly distributed in $[0,1]$. 
In particular, 
the \emph{random simple temporal graph} model is obtained by adding such labels to the edges of an Erd\H{o}s-R\'{e}nyi random graph $G_{n,p}$. Random simple temporal graphs exhibit some remarkable phase transitions(see, e.g., \cite{angel2020,mertzios2024,becker2022,broutin2023,casteigts2024,atamanchuk2024size}).

Since (sparse) Erd\H{o}s-R\'{e}nyi random graphs are locally tree-like, it is natural to study analogous random trees. This motivates our definition of a \emph{uniform temporal tree}, specified below. The paper's main goal is to study the basic properties of such random trees, including their size, height, and degree distribution. 

The definition is the following.
For a positive integer $n$, let
$T_{n}$ be a rooted infinite complete $n$-ary tree (i.e., the root vertex has degree $n$ and every other vertex has degree $n+1$). To each edge $e$ of $T_n$, assign an independent
random variable $U_e$, uniformly distributed in $[0,1]$.
$U_e$ is called the \emph{label} of the edge $e$. 
A path between the root and a vertex $v$ is called \emph{decreasing} if the edge labels in the path appear in decreasing order. 
Sometimes it is more convenient 
to assign labels to vertices. 
The label $\ell_v$ of a vertex $v \in T_n$ is the label $U_e$ of 
the edge $e$ connecting $v$ to its parent (the parent of a vertex 
$v$ is the vertex adjacent to $v$ that is on the path between the root and $v$). In some cases we focus on the vertex labels, calling a path decreasing if the vertex labelling is decreasing, though this is an equivalent definition. Note that the root vertex does not have a parent or a label.
 The ($p$-percolated) uniform temporal tree, $\mathcal{T}_{n,p}$, is a random tree obtained from $T_{n}$ by assigning the root the label $p$ and deleting all vertices whose path from the root is not decreasing with respect to the vertex labelling. Note that every vertex in a $\mathcal{T}_{n,p}$ has a label at most $p$.
 When $p=1$, we simplify the notation and just write $\mathcal{T}_n$ for $\mathcal{T}_{n,1}$. 

It is clear that $\mathcal{T}_n$ (and therefore $\mathcal{T}_{n,p}$) is almost surely finite. Indeed, the probability that $\mathcal{T}_{n}$ has a vertex at depth $k$ is at most $n^k/k!$ which goes to zero as $k\to \infty$.

In the next section, we present the main results of the paper concerning the distribution of the size of $\mathcal{T}_{n,p}$, the typical depth of a vertex, the height (i.e., the depth of the deepest vertex), and the degree distribution. The proof of these results is given in subsequent sections.


\section{Results}

Before presenting the main results, we fix some terminology and notation. Let $T$ be a rooted tree. $P(v)$ for any $v \in T$ (except for the root) denotes the unique path between the root and $v$. The set $P^-(v)$ is $P(v)$ with the root removed. The \emph{depth} of $v$, $|v|$, is the number of edges in $P(v)$. The \emph{parent} of $v$, $p(v)$ is the single neighbour of $v$ in $P(v)$, and the set of \emph{children} of $v$, $C(v)$, contains all vertices at depth $|v|+1$ that are adjacent to $v$. 
The \emph{out-degree} of a vertex $v$ is the number $|C(v)|$ of its children. Two vertices $u$ and $v$ are \emph{siblings} if $p(u) = p(v)$. For a vertex $v \in T$, $T(v)$ is the subtree of $T$ rooted at $v$ containing all \emph{descendants} of $v$ in $T$, i.e., the tree containing $v$, its children, grandchildren, and so on.
Convergence in distribution for a sequence of random variables is denoted by $\convdist$ and $\dist$ is used for equality in distribution. Finally, we let $\convprob$ represent convergence in probability for a sequence of random variables.

The first result concerns the size of $\mathcal{T}_{n,p}$.
It is not difficult to see that the expected size equals $e^{np}$. However, the size
does not concentrate around the mean. We show that it admits a limit law, in the sense that the size divided by its expectation converges in distribution to an exponential random variable.
Moreover, we establish a joint limit law for the 
"distribution of mass" at the root, that is, for the sizes of the subtrees of children of the root with the largest labels.

\begin{theorem}\label{sizethm}
 Let $p\in (0,1]$ and consider a percolated uniform temporal tree. Then 
 \[
 \EXP |\mathcal{T}_{n,p}| = e^{np}
 \]
 and
     $$\frac{|\mathcal{T}_{n,p}|}{e^{np}} \convdist E \ \text{ as }  n \to \infty~,$$
where $E$ is an exponential$(1)$ random variable.

Moreover, for $1 \leq i \leq n$, let $v_i$ be the child of the root with the $i$-th largest label. Then
for any fixed $m\ge 1$,
$$\left(\frac{|\mathcal{T}_{n,p}(v_1)|}{e^{np}}, \ldots , \frac{|\mathcal{T}_{n,p}(v_m)|}{e^{np}}\right) \convdist \left( E_1U_1 , E_2U_1U_2 , \ldots , E_mU_1 \cdots U_m \right) \ \text{ as } n \to \infty~,$$
where $(E_k)_{k \geq 0}$ is a sequence of independent $\exponential(1)$ random variables and $(U_k)_{k \geq 0}$ is an independent sequence of independent uniform random variables on $[0,1]$.       
\end{theorem}

\begin{remark}
It follows from Theorem \ref{sizethm} that if $U_1,U_2,\ldots$ are independent uniform $[0,1]$ and $E_1,E_2,\ldots$ are independent
exponential$(1)$ random variables, then
$$
\sum_{i=1}^\infty E_i\prod_{j=1}^i U_j
$$
is an exponential$(1)$ random variable. This identity may also be checked directly.
\end{remark}

The proof of Theorem \ref{sizethm} is given in Section \ref{size}. The proof relies on the theory of branching random walks. Two key ingredients of the argument are the following:

\begin{enumerate}
    \item For any $p \in (0,1]$, $\ex|\mathcal{T}_{n,p}|^2 = O((\ex|\mathcal{T}_{n,p}|)^2)$.     
    \item Any $\mathcal{T}_{n,p}$ can be decomposed, for any $L \geq 1$, into a collection of smaller trees $(\mathcal{T}_{n,p_i})_{i=1}^{2^L}$, where the vector $(p_1,\ldots,p_{2^L})$ has a distribution that can be described using the values of vertices in the $L$-th generation of a branching random walk.
\end{enumerate}


The next result concerns the \emph{height} $H_{n,p}$ of a uniform temporal tree, that is, the maximum vertex depth in $\mathcal{T}_{n,p}$. The following theorem, proved in Section \ref{sec:height}, states that the height is about $e$ times the logarithm of the size of the tree. This property is reminiscent of uniform random recursive trees (see \cite{Devroye1987}).

\begin{theorem}\label{heightthm}
    Fix $p \in (0,1]$, and let $H_{n,p}$ denote the height of a percolated uniform temporal tree $\mathcal{T}_{n,p}$. Then
    
        $$\frac{H_{n,p}}{np} \convprob e
        \ \ \text{and} \ \
        \frac{H_{n,p}}{\log|\mathcal{T}_{n,p}|} \convprob e
        \ \text{ as }  n \to \infty.$$
\end{theorem}

Note that the second statement follows from the first since Theorem \ref{sizethm} implies that
$(\log|\mathcal{T}_{n,p}|)/(np) \to 1$ in probability.
The proof of this theorem is based on a connection between uniform temporal trees and branching random walks that is similar to the one used for Theorem \ref{sizethm}.

The next property establishes the typical depth of a vertex in a uniform temporal tree. Just like in a uniform random recursive tree, the depth is concentrated around the natural logarithm of the size of the tree. The proof is provided in Section \ref{sec:depth}.

\begin{theorem}\label{depththm}
    Let $p \in (0,1]$, and let $D_{n,p}$ denote the depth of a uniformly chosen vertex in a percolated uniform temporal tree $\mathcal{T}_{n,p}$. Then    
        $$\frac{D_{n,p}}{np} \convprob 1 \ \ \text{and} \ \
        \frac{D_{n,p}}{\log|\mathcal{T}_{n,p}|} \convprob 1
        \ \text{ as }  n \to \infty.$$
\end{theorem}

Finally, the next theorem establishes the asymptotic expected degree distribution of a uniform temporal tree. In particular, the expected number of leaves is about half of the expected number of vertices, the expected number of vertices with one child is a quarter of the expected size, etc. Once again, this property is similar to the corresponding asymptotic degree distribution of a uniform random recursive tree.
The following theorem is proved in Section \ref{sec:degrees}.

\begin{theorem}\label{degreethm}
    Let $p \in (0,1]$, and, for $k\ge 0$, let $L_{n,k}$ denote the number of vertices of out-degree $k$ in a percolated uniform temporal tree $\mathcal{T}_{n,p}$. Then    
        $$\frac{\EXP L_{n,k}}{e^{np}} \to 2^{-(k+1)} \ \text{ as }  n \to \infty.$$
\end{theorem}

As mentioned above, uniform temporal trees share several of their basic characteristics with random recursive trees.
The uniform random recursive tree on $n$ vertices is a random rooted tree with vertices labelled in $\{1,\ldots,n\}$. The root has label $1$. 
Vertices $i\in \{2,\ldots,n\}$ are attached recursively such that vertex $i$ is attached to a vertex in $\{1,\ldots,i-1\}$ selected uniformly at random. The uniform random recursive tree is one of the most ubiquitous trees in computer science and has been thoroughly studied (see \cite{Meir1978,Devroye1988,Pittel1994,Janson2005,Drmota2009,Addario2018}). 

One may wonder whether a random recursive tree is equivalent to a uniform temporal tree $\mathcal{T}_{n}$ when conditioning on the size $|\mathcal{T}_{n}|$. However, while in a uniform temporal tree, the root vertex always has maximal degree, in a uniform random recursive tree, there are vertices with much higher degree \cite{devroye1995strong,Addario2018,eslava2022depth}. Moreover, the distribution of mass at the root established in Theorem \ref{sizethm} is different from the "stick-breaking" distribution of the uniform random recursive tree.


\cite{broutin2023} utilize direct couplings between neighbourhoods of vertices in sparse random simple temporal graphs and uniform random recursive trees to prove statements about connectivity. In sparse random graphs neighbourhoods around vertices are tree-like and so match the structure of the $\mathcal{T}_{n,p}$ closely.


\section{The uniform spacings coupling}
\label{sec:spacings}

In the proof of Theorems \ref{sizethm}
and \ref{heightthm} we use representations of the labels in a uniform temporal tree as a sum of uniform spacings. We explore the connection in this section. 

Let $U_1,\ldots,U_n$ be a collection of independent $\unif[0,1]$ random variables and let $V_1 \geq \cdots \geq V_n$ be the
corresponding order statistics. We also set $V_0 = 1$, $V_{n+1} = 0$. 
Writing $S_i=V_{i-1}-V_i$ ($i\in \{1,\ldots,n+1\}$) for the induced spacings,
we use the representation 
$$
    \left( S_{1} , \ldots , S_{n+1}\right) := \left(V_0-V_1 , \ldots , V_{n}-V_{n+1}\right) \dist \left(\frac{E_1}{\sum_{i=1}^{n+1}E_i} , \ldots , \frac{E_{n+1}}{\sum_{i=1}^{n+1} E_i} \right),
    $$
where $E_1,\ldots,E_{n+1}$ are independent $\exponential(1)$ random variables (see, e.g., \cite{Devroye1986book}). We record a key observation about uniform random variables needed for the incorporation of uniform spacings. 

    \begin{lemma}\label{spacings}
        Let $U_1,\ldots,U_n$ be a collection of independent $\unif[0,1]$ random variables with corresponding spacings $S_1,\ldots,S_{n+1}$,  let $x \in [0,1]$, and let $I = \{1 \leq i \leq n : U_i \geq x\}$. Define a collection of random variables $(V_1,\ldots,V_n)$, where
        $$
        V_i = \left\{ 
        \begin{array}{cc}
        U_i-1     & i \in I \\
        U_i     & \text{otherwise}
        \end{array}
        \right.
        $$
        Then, $(V_1,\ldots,V_n)$ is distributed like a vector of independent uniform random variables on the interval $[x-1,x]$ i.e., $(V_1,\ldots,V_n) \dist (U_1,\ldots,U_n) - (1-x)$. In particular, if $(S_1,\ldots,S_{n+1})$ is a vector of uniform spacings, then
        $$
        (V_1 , \ldots , V_n) \dist (x-S_1 , \ldots , x-(S_1 + \cdots + S_n)).
        $$
    \end{lemma}

\begin{proof}
The transformation described in the lemma is equivalent to moving the section of the interval $[0,1]$ above $x$ (and the corresponding points in $\{U_1,\ldots,U_n\}$ that lie above $x$) to be below zero. The points in $[0,x]$ and $[1-x,0]$ are still uniformly distributed over these intervals, and thus together are uniform over $[1-x,x]$.
\begin{figure}[H]
\centering
\begin{tikzpicture}[scale=0.5]
    \draw [thick] (-10,0)-- (6,0);
    \draw [ultra thick,blue] (6,0)-- (10,0);
    \draw [thick] (-10,0.5) -- (-10,-0.5) node[below] {0};
    \draw [thick] (10,0.5) -- (10,-0.5) node[below] {1};
    \draw [thick] (6,0.5) -- (6,-0.5) node[below] {$x$};
    \fill (3,0) circle (0.25);
    \draw [<->,thick] (6,0.5) -- (3,0.5) node[above] {$S_1$};
    \fill (0,0) circle (0.25);
    \fill (-1,0) circle (0.25);
    \fill (9,0) circle (0.25);

    \draw [thick] (-10,-4)-- (6,-4);
    \draw [ultra thick,blue] (-10,-4)-- (-14,-4);
    \draw [thick] (-10,-3.5) -- (-10,-4.5) node[below] {0};
    \draw [thick] (6,-3.5) -- (6,-4.5) node[below] {$x$};
    \draw [thick] (-14,-3.5) -- (-14,-4.5) node[below] {$x-1$};
    \fill (3,-4) circle (0.25);
    \draw [<->,thick] (6,-3.5) -- (3,-3.5) node[above] {$S_1$};
    \fill (0,-4) circle (0.25);
    \fill (-1,-4) circle (0.25);
    \fill (-11,-4) circle (0.25);
    \draw [->,thick,dashed] (3,0) -- (3,-2);
    \draw [->,thick,dashed] (0,0) -- (0,-2);
    \draw [->,thick,dashed] (-1,0) -- (-1,-2);
\end{tikzpicture}
\caption{The rotation of the uniform spacings around a vertex $x$. The blue section above $x$ is moved from above $x$ to below $0$. After the segment is moved the points are distributed uniformly over $[x-1,x]$.}
\label{fig:spacings}
\end{figure}
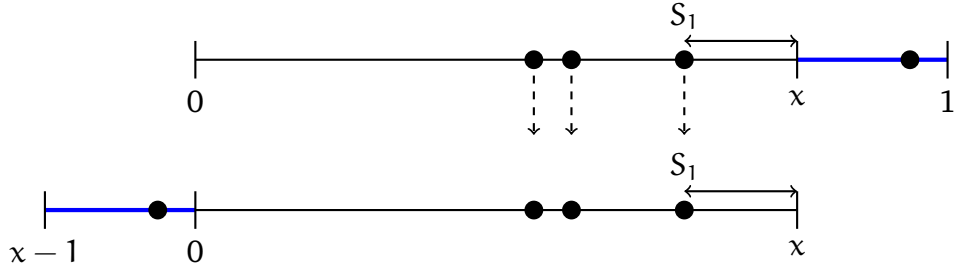
    
\end{proof}

Using the above fact and the uniform spacings representation we obtain an equivalent description of how labels evolve in uniform temporal trees. Let $((S_{v,1} , \ldots , S_{v,n+1}) : v \in T_n)$ be a collection of independent uniform spacings. For a vertex $v \in T_n$, with label $\ell_v$, we define the labels of its children $v_1,\ldots,v_n$ as follows:
    $$
    \ell_{v_j} = p - \sum_{i=1}^j S_{v,i} \quad \forall j \in \{1,\ldots,n\},
    $$
where we delete vertices with negative labels in $\mathcal{T}_{n,p}$. This is equivalent to deleting the vertices that have labels above $p$, and by applying the reverse of the transformation from the above lemma, we see that this labeling is equivalent to that of $T_n$. 
We refer to
generating the labels of $T_n$ in this manner  as the \emph{uniform spacings coupling}. The following notation is useful when discussing the evolution of labels under the uniform spacings coupling. 

\begin{enumerate}
    \item Define the \emph{rank} of a vertex in a $\mathcal{T}_{n,p}$ as the placement of its label among its siblings. That is, the sibling with the largest label gets rank $1$, the sibling with the second largest label gets $2$ etc. We denote the rank of a vertex by $r(v)$, and by convention, the rank of the root is $0$.
    \item The \emph{index} of a vertex is the sum of all the ranks of vertices in $P(v)$, $\iota(v) = \sum_{u \in P(v)} r(u)$.
    \item The set $I(j)$ for any $j \geq 0$ is the subset of all vertices in a $\mathcal{T}_{n,p}$ that have index $j$, $I(j) = \{v \in \mathcal{T}_{n,p} : \iota(v) = j\}$.
\end{enumerate}

Using this notation, we may describe the label of any vertex $v$ that exists in a $\mathcal{T}_{n,p}$ in terms of ranks via the uniform spacings coupling: 
    \begin{equation}\label{labelcoupling2}
    \ell_{v} = p - \sum_{u \in P^-(v)}\sum_{i=1}^{r(u)}S_{p(u),i}~.
    \end{equation}    
See Figure \ref{fig:spacings2} for an illustration.
We use this representation in the proofs of Theorems \ref{sizethm} 
and \ref{heightthm}.

\begin{figure}[H]
\label{fig:spacings2}
\centering
\footnotesize{
\begin{forest}
[$p$
    [$p-S_1$
        [$p-S_1-S_1'$]
        [$p-S_1-S_1'-S_2'$]
        [$\cdots\cdots$]
    ]
    [$p-S_2$]
    [$\cdots\cdots$]
]
\end{forest}
}
\caption{The evolution of labels in a $\mathcal{T}_{n,p}$ according to the spacings coupling. The random variables $S_1,S_2,S_1',S_2'$ are all uniform spacings. The label of a vertex is the label of its next lower-rank sibling (or parent if its rank is $1$) minus a spacing.}
\label{fig:coupling}
\end{figure}
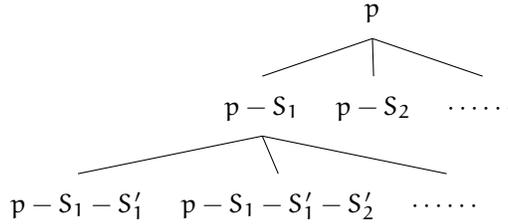

\section{The size}\label{size}

Throughout this section, $p$ is always a fixed parameter in $(0,1]$. First, we examine the first two moments of the size of a $\mathcal{T}_{n,p}$. There are $n^k$ vertices at depth $k$ in $T_n$, and each $v \in T_n$ with $|v| = k$ exists in $\mathcal{T}_{n,p}$ with probability $p^k/k!$ as the edges in $P(v)$ must all be below $p$ and monotone decreasing. Hence, for any $n \geq 1$,
    $$    \ex|\mathcal{T}_{n,p}| = \sum_{k=0}^\infty \frac{(np)^k}{k!} = e^{np}~,
    $$
establishing the first statement of Theorem \ref{sizethm}.
In the next lemma, we provide an upper bound for the second moment. The upper bound yields a concentration inequality for the average size of large collections of similarly distributed subtrees in a uniform temporal tree $\mathcal{T}_{n,p}$.

\begin{lemma}\label{tightness} $\empty$
    \begin{enumerate}
        \item For all $n$, $\ex|\mathcal{T}_{n,p}|^2 \leq 5\left(\ex|\mathcal{T}_{n,p}|\right)^2$. 
        \item Let $m$ be a fixed positive integer and let $q_1,\ldots,q_m \in (0,p)$. For $i\in [m]$, define $p_i = p-q_i$.
        Let $\mathcal{T}_{n,p_i}$, $i\in [m]$ be independent uniform temporal trees and
        let $\mu = \sum_{i=1}^m \ex|\mathcal{T}_{n,p_i}| = \sum_{i=1}^m e^{np_i}$.
        For all $\epsilon > 0$ and $n$,
        $$\prob\left(\left|\sum_{i=1}^m |\mathcal{T}_{n,p_i}| - \mu\right| > \epsilon\mu \right) \leq \frac{5}{\epsilon^2}\left(\frac{\max_{1 \leq i \leq m} e^{-nq_i}}{\sum_{i=1}^m e^{-nq_i}}\right).$$
    \end{enumerate}
\end{lemma}

The proof of the lemma is given in the Appendix.
As mentioned in the introduction, aside from the above concentration inequality, the other main ingredient in the proof of Theorem \ref{sizethm} is a connection between the labels of $\mathcal{T}_{n,p}$ and a branching random walk. The representation in terms of uniform spacings described in Section \ref{sec:spacings} plays a key role in the construction.

First, let us recall the definition of a branching random walk (see  \cite{Shi2015} for a general introduction to branching random walks). Given a random variable $X$ (called the \emph{step size}) and a locally finite rooted tree $T$, label each edge in $T$ with an independent copy of $X$, $(X_e)_{e \in T}$. The branching random walk on $T$ is a collection of vertex-indexed random variables $(Y_v)_{v \in V}$, where $Y_v$ is the sum of the labels of the edges in $P(v)$. The random variable $Y_v$ is called the \emph{value} of the vertex $v$ in the branching random walk.

We define a branching random walk on a rooted infinite complete binary tree $\mathcal{T}^*$. The root has one child, and every other vertex has precisely two children. To match the standard notation for branching random walks, the root is in generation $-1$, its child in generation $0$, etc. For any generation $L\ge 0$, there are $2^L$ vertices.

\begin{lemma}\label{branchrandwalk}
    Let $L \geq 0$ be an integer and let $\epsilon,\delta,x > 0$. Let $p-p_1,\ldots,p-p_{2^L}$ be the values of the vertices in the $L$-th generation of a branching random walk on the infinite complete binary tree $\mathcal{T}^*$ with step size $X \dist \frac{E}{n}$, where $E$ is an exponential$(1)$ random variable. Define $p_i^+(\epsilon)$ and $p_i^-(\epsilon)$ such that $p-p_i^+(\epsilon) = (1+\epsilon)(p-p_i)$ and $p-p_i^-(\epsilon) = (1-\epsilon)(p-p_i)$ for all $1 \leq i \leq 2^L$. Then,
    $$ 
        \prob\left(\frac{1}{e^{np}}|\mathcal{T}_{n,p}| > x \right) \leq \prob\left(\frac{1}{e^{np}}\left(\sum_{i=1}^{2^L} |\mathcal{T}_{n,p_i^-(\epsilon)}| + |\mathcal{T}_{n,p_i^-(\epsilon)}'|\right) > x(1-\delta) \right) + o_n(1)
    $$
    and
    $$
        \prob\left(\frac{1}{e^{np}}|\mathcal{T}_{n,p}| > x \right) \geq \prob\left(\frac{1}{e^{np}}\left(\sum_{i=1}^{2^L} |\mathcal{T}_{n,p_i^+(\epsilon)}| + |\mathcal{T}_{n,p_i^+(\epsilon)}'|\right) > x(1+\delta) \right) - o_n(1)~,
    $$
    where the trees $\mathcal{T}_{n,p_i^\pm(\epsilon)},\mathcal{T}_{n,p_i^\pm(\epsilon)}'$ are all conditionally independent given ($p_1,\ldots,p_{2^L}$).
\end{lemma}

\begin{proof}[Proof of Lemma \ref{branchrandwalk}]

Let $\cE$ be the event that all vertices with $\iota(v) \leq L$ have children of index $(\iota(v)+1),\ldots,L+1$. That is, $\cE$ is the event that the out-degree of each vertex in $I(j)$ is at least $L-j+1$ for each $j \in \{0,\ldots,L\}$. Since the degree of fixed-index vertices tends to infinity as $n \to \infty$ and $L$ is a fixed integer, we know that $\prob(\cE) \to 1$ as $n \to \infty$.

Let $v \in \mathcal{T}_{n,p}$ be a fixed vertex of rank $k$, and suppose that $p(v)$ has $\ell$ children listed in order of increasing rank $v_1,\ldots,v_\ell$ where $v = v_k$. We define $\mathcal{F}(v)$ to be the collection of all the subtrees rooted at higher rank siblings of $v$, that is,

$$\mathcal{F}(v) = \bigcup_{i=k+1}^\ell \mathcal{T}_{n,p}(v_i).$$

A simple observation is that every vertex in $\cup_{j \geq L+1} I(j)$ is in exactly one set $\mathcal{T}_{n,p}(v)$ or $\mathcal{F}(v)$ for some $v \in I(L+1)$, that is,
$$
    \bigcup_{j \geq L+1} I(j) = \bigcup_{v \in I(L+1)} \left(\mathcal{T}_{n,p}(v) \cup \mathcal{F}(v)\right).
$$
Since all the sets on the right-hand side are disjoint and $|I(0) \cup \ldots \cup I(L+1)|$ is finite, this implies that

\begin{equation}\label{sizeidentity}
    \frac{1}{e^{np}}|\mathcal{T}_{n,p}| = \frac{1}{e^{np}}\sum_{v \in I(L+1)} |\mathcal{T}_{n,p}(v)| + \frac{1}{e^{np}}\sum_{v \in I(L+1)} |\mathcal{F}(v)| + o_n(1).
\end{equation}
For any fixed $k \geq 0$, define $\mathcal{T}_{n,p}(k)$ to be the induced subtree of $\mathcal{T}_{n,p}$ on the vertex set $I(0) \cup \cdots \cup I(k)$. Let $\epsilon > 0$. Since $\mathcal{T}_{n,p}(L+1)$ is a finite tree, we can apply the law of large numbers to the spacings from the uniform spacings coupling $(S_{v,i} : v \in \mathcal{T}_{n,p}(L+1), i \in \{1,\ldots,L+1\})$ to assert that, for a sequence of independent $\exponential(1)$ random variables $(E_{v,i})_{v \in T_n, 1 \leq i \leq n+1}$,
    $$
    \prob(S_\leq) := \prob\left( \bigcap_{v \in \mathcal{T}_{n,p}(L+1)}\left\{ \ell_v \leq \ell_v^-(\epsilon) \right\} \right) \to 1
    $$
and
    $$
    \prob(S_\geq) := \prob\left( \bigcap_{v \in \mathcal{T}_{n,p}(L+1)}\left\{ \ell_v \geq \ell_v^+(\epsilon) \right\} \right) \to 1
    $$
as $n \to \infty$, where
    $$
    \ell_v^-(\epsilon) := p - \frac{1-\epsilon}{n}\sum_{u \in P^-(v)}\sum_{i=1}^{r(u)} E_{p(u),i}, \text{ and } \ell_v^+(\epsilon) := p - \frac{1+\epsilon}{n}\sum_{u \in P^-(v)}\sum_{i=1}^{r(u)} E_{p(u),i} .
    $$
Conditioning on $S_\leq$ and $S_\geq$ allows us to remove many dependencies between vertex labels.

Next, on the event $\cE$, we recursively define a one-to-one mapping $\phi$ of the vertices in a $\mathcal{T}_{n,p}(L+1)$ with the vertices of $\mathcal{T}^*(L)$ (the tree truncated at generation $L$) and assign a corresponding edge labeling $(X_e)_{e \in \mathcal{T}^*(L)}$. The following properties hold for our mapping and together provide our branching random walk description:
\begin{enumerate}
    \item In $\mathcal{T}^*(L)$, the collection of random variables $(\sum_{e \in P(v)} X_e)_{v \in \mathcal{T}^*(L)}$ is a branching random walk with step size $\frac{1+\epsilon}{n}E$.
    \item For all $v \in \mathcal{T}_{n,p}(L+1)$, the value of $\phi(v)$ is $\ell_v^+(\epsilon)$, i.e., $\sum_{e \in P(\phi(v))} X_e = \ell_v^+(\epsilon)$.
    \item The $j$th generation of $\mathcal{T}^*(L)$ contains the vertices of index $j+1$ in $\mathcal{T}_{n,p}(L+1)$.
\end{enumerate}
We only describe the construction for $(\ell^+_v(\epsilon) : v \in I(L+1))$, though the same can be done for $(\ell^-_v(\epsilon) : v \in I(L+1))$ using a similar procedure.

First, the root of $\mathcal{T}_{n,p}(L+1)$, which we again denote with $\rho$, is mapped to the root of $\mathcal{T}^*(L)$. The unique index-1 vertex in $\mathcal{T}_{n,p}(L+1)$---call it $c$---, is mapped to the unique child of the root in $\mathcal{T}^*(L)$. The edge going into this vertex in $\mathcal{T}^*(L)$ is given the label $X_e = \frac{1+\epsilon}{n}E_{\rho,1}$. The three properties above hold for this base case.

Now, suppose that $\phi$ has been defined for $I(0),\ldots,I(k)$, $k \leq L$ and that the three properties hold for the partial assignment. Each vertex in $v \in I(k)$ has exactly one child $c_v$ and one sibling $s_v$ of index $k+1$. As $v \in I(k)$, its image $\phi(v)$ in $\mathcal{T}^*$ is already defined, and by assumption the left and right children of $\phi(v)$, $\phi(v)_\ell$ and $\phi(v)_r$, are not in the image $\phi(I(0) \cup \cdots \cup I(k))$. Define $\phi(c_v) = \phi(v)_\ell$ and $\phi(s_v) = \phi(v)_r$. We give the edges $e_1 = \{\phi(v),\phi(v)_\ell\}$ and $e_2 = \{\phi(v),\phi(v)_r\}$ the labels $\frac{1+\epsilon}{n}E_{v,1}$ and $\frac{1+\epsilon}{n}E_{p(v),r(v)+1}$ respectively. Point (iii) holds by definition for this extension of $\phi$ as all index-$(k+1)$ vertices are either a direct sibling or child of an index-$k$ vertex previously placed into the tree. Moreover, by assuming that (ii) holds for the $k$-th step implies that it still holds for the $k+1$-th step. Since the exponential random variables $E_{v,1}$ and $E_{p(v),r(v)+1}$ have not been used in the construction previously (this is a consequence of assuming that property (ii) holds for the first $k$ levels), property (i) holds for the extension as well.

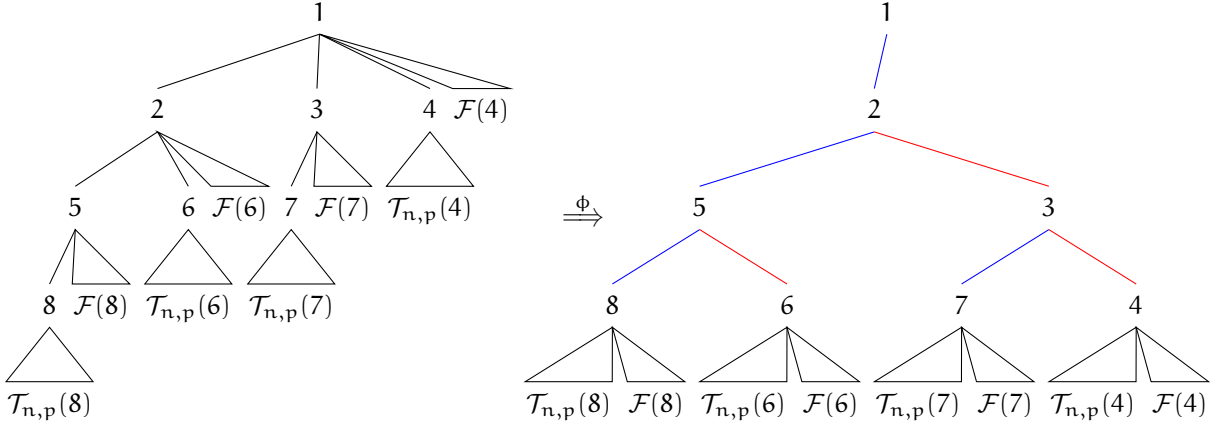
\begin{figure}[H]
\centering
\label{fig:bij}
\footnotesize{
\begin{forest}
for tree={s sep=2mm, inner sep=0, l*=1.2, thick}
[,phantom ,for children={l sep=1ex,fit=band}
    [$1$
        [$2$
            [$5$
                [$8$
                    [$\mathcal{T}_{n,p}(8)$,roof]
                ]
                [$\mathcal{F}(8)$,roof]
            ]
            [$6$
                [$\mathcal{T}_{n,p}(6)$,roof]   
            ]
            [$\mathcal{F}(6)$,roof]
        ]
        [$3$
            [$7$
                [$\mathcal{T}_{n,p}(7)$,roof]
            ]
            [$\mathcal{F}(7)$,roof]
        ]
        [$4$
            [$\mathcal{T}_{n,p}(4)$,roof]
            { \draw (.east) node[right]{$ \quad \quad \quad {\overset{\phi}{\Longrightarrow}}$}; }
        ]
        [$\mathcal{F}(4)$,roof]
    ]
    [$1$
        [$2$,edge=blue
            [$5$,edge=blue
                [$8$,edge=blue
                    [$\mathcal{T}_{n,p}(8)$,roof]  
                    [$\mathcal{F}(8)$,roof]            
                ]
                [$6$,edge=red
                    [$\mathcal{T}_{n,p}(6)$,roof]  
                    [$\mathcal{F}(6)$,roof]            
                ]
            ]
            [$3$,edge=red
                [$7$,edge=blue
                    [$\mathcal{T}_{n,p}(7)$,roof]  
                    [$\mathcal{F}(7)$,roof]
                ]
                [$4$,edge=red
                    [$\mathcal{T}_{n,p}(4)$,roof]  
                    [$\mathcal{F}(4)$,roof]
                ]
            ]
        ]
        [,phantom]
    ]
]
\end{forest}
}
\caption{The mapping $\phi$ up to index $L=3$. The left tree is $\mathcal{T}_{n,p}$ and the right is the binary tree $\mathcal{T}^*$ with the labelling obtained from $\phi$. The vertices are ordered from left to right in order of increasing index in $\mathcal{T}_{n,p}$. A left child (blue edge) in $\mathcal{T}^*$ corresponds to moving down to the vertex's  child of the smallest index in $\mathcal{T}_{n,p}$, and a right child (red edge) corresponds to moving to a vertex's sibling of the smallest index in  $\mathcal{T}_{n,p}$.}
\end{figure}



This construction is almost sufficient to complete the proof, though there is still an approximation for $\mathcal{F}(v)$ with a tree that is needed. We delay proving this fact until the appendix, though we record and use the result.

\begin{lemma}\label{treeforest}
    Let $\delta > 0$ and let $v \in \mathcal{T}_{n,p}$ be a vertex of fixed finite index. There are random trees $\mathcal{T}^-(v) \dist |\mathcal{T}_{n,\ell_v^-(\epsilon)}|$, $\mathcal{T}^+(v) \dist |\mathcal{T}_{n,\ell_v^+(\epsilon)}|$, conditionally independent of the labels of $v$ and its lower rank siblings in $T_n$ given $(\ell_v^\pm(\epsilon) : v \in I(L+1))$ (and in particular $T_n(v)$), such that $\prob(|\mathcal{F}(v)| \leq (1+\delta)|\mathcal{T}^-(v)|) \to 1$ and $\prob(|\mathcal{F}(v)| \geq (1-\delta)|\mathcal{T}^+(v)|) \to 1$ as $n \to \infty$.
\end{lemma}

Note in the above lemma that $\mathcal{F}(v)$ is already independent of everything other than the label of its siblings and its parent, so the above-noted independence lines up with the desired result. The approximation for $\mathcal{T}_{n,p}(v)$ is a little bit simpler. If we remove paths that are non-decreasing and have labels all below $\ell_v^-(\epsilon)$ instead of $\ell_v$ from $T_n$, then the resulting tree $\mathcal{T}_{approx}(v)$ is at least as large as $\mathcal{T}_{n,p}(v)$ and is conditionally independent of the rest of $T_n$ when we condition on the label $\ell_v^-(\epsilon)$ as desired.

Using this lemma, and the fact that the events $S_\leq$ and $\cE$ have probability tending to $1$ as $n \to \infty$ we see that, for any $x > 0$ and any $\delta > 0$, (\ref{sizeidentity}) gives
    \begin{align*}
        \prob\left(\frac{1}{e^{np}}|\mathcal{T}_{n,p}| >x \right) &\leq \prob\left(\frac{1}{e^{np}}\left(\sum_{v \in I(L+1)} |\mathcal{T}_{n,p}(v)| + |\mathcal{F}(v)|\right) + o_n(1) > x \right) + o_n(1) \\
        &\leq \prob\left(\frac{(1+\delta)}{e^{np}}\left(\sum_{v \in I(L+1)} |\mathcal{T}_{approx}^-(v)| + |\mathcal{T}^-(v)|\right) > x \right) + o_n(1).
    \end{align*}
Replacing $\mathcal{T}^-_{approx}(v)$ and $\mathcal{T}^-(v)$ in the above computation with $\mathcal{T}^+_{approx}(v)$ and $\mathcal{T}^+(v)$ yields the corresponding lower bound, completing the proof.
\end{proof}

Applying Lemmas \ref{tightness} and \ref{branchrandwalk} transforms the analysis of $|\mathcal{T}_{n,p}|$ into the analysis of a branching random walk on $\mathcal{T}^*$. Next, we establish some results about the aforementioned walk.

\begin{lemma}\label{finalcomputation} Let $(p_i)_{i=1}^{2^L},(p_i^\pm(\epsilon))_{i=1}^{2^L}$ be as in Lemma \ref{branchrandwalk} and set $q_i = p-p_i$, $q_i^\pm(\epsilon) = p-p_i^\pm(\epsilon)$. 
Introduce the notation 
$Q_i=nq_i$ and $Q^\pm_{i,\epsilon}=n q_i^\pm(\epsilon)$.
For any $n \geq 1$,
    \begin{enumerate}
        \item $X_L := \sum_{i=1}^{2^L} e^{-Q_i} \convdist \frac{E}{2}$ as $L \to \infty$.
        \item Let $X_L^\pm(\epsilon) := \sum_{i=1}^{2^L} e^{-Q^\pm_{i,\epsilon}}$. There is a sequence $\epsilon(L)$ for $L \geq 0$ such that $X_L^\pm(\epsilon(L)) \convdist \frac{E}{2}$.
        \item There is a sequence $\epsilon(L)$ for $L \geq 0$ such that $\ex\left[\frac{\max_{1 \leq i \leq 2^L} \exp(-Q^\pm_{i,\epsilon(L)})}{X_L^\pm(\epsilon(L))}\right] \to 0$ as $L \to \infty$. Moreover, $\ex\left[\frac{\max_{1 \leq i \leq 2^L} \exp(-Q_i))}{X_L}\right] \to 0$ as $L \to \infty$
    \end{enumerate}
\end{lemma}

\begin{remark}
    Note that the random variables $X_L$ and $X_L^\pm(\epsilon)$ do not depend on the parameter $n$. The step sizes of the branching random walks defining each $p_i$ and $p_i^\pm(\epsilon)$ are distributed like $\frac{1}{n}E$ and $\frac{1\pm\epsilon}{n}E$ respectively, so the random variables $Q_i=nq_i$ and $Q^\pm_{i,\epsilon}= nq_i^\pm(\epsilon)$ do not depend on $n$ due to cancellations.
\end{remark}

All of the required lemmas have now been presented and we are prepared to prove
 Theorem \ref{sizethm}.

\begin{proof}[Proof of Theorem \ref{sizethm}]
Let $x > 0$. By Lemma \ref{branchrandwalk} we know that, for any $\epsilon,\delta >0$ and $L \geq 0$
$$
\prob\left(\frac{1}{e^{np}}|\mathcal{T}_{n,p}| > x \right) \leq \prob\left( \frac{1}{e^{np}}\sum_{i=1}^{2^L} \left( |\mathcal{T}_{n,p_i^-(\epsilon)}| + |\mathcal{T}_{n,p_i^-(\epsilon)}'| \right) > (1-\delta)x\right) + o_n(1).
$$
Applying Lemma \ref{tightness} we get that for any $\delta_2 > 0$,
    \begin{align*}
    & \prob\left(\frac{1}{e^{np}}|\mathcal{T}_{n,p}| > x\right) \\
    \leq & \prob\left(\frac{1}{e^{np}}\sum_{i=1}^{2^L}\left(e^{n(p-q_i^-(\epsilon))} + e^{n(p-q_i^-(\epsilon))}\right) > \frac{1-\delta}{1+\delta_2}x\right) + \frac{5}{\delta_2^2}\ex\left[ \frac{\max_{1 \leq i \leq 2^L}e^{-Q_i}}{\sum_{i=1}^{2^L} e^{-Q_i}} \right] + o_n(1) \\
    = &  \prob\left(\sum_{i=1}^{2^L}2e^{-Q^-_{i,\epsilon}} > \frac{1-\delta}{1+\delta_2}x\right) + \frac{5}{\delta_2^2}\ex\left[ \frac{\max_{1 \leq i \leq 2^L}e^{-Q_i}}{\sum_{i=1}^{2^L} e^{-Q_i}} \right] + o_n(1).  
\end{align*}
Note that this upper bound does not have any dependence on $n$, and so implies that 
    \begin{align*}
    \limsup_{n \to \infty} \prob\left(\frac{1}{e^{np}}|\mathcal{T}_{n,p}| > x\right) \leq \prob\left(\sum_{i=1}^{2^L}2e^{-Q^-_{i,\epsilon}} > \frac{1-\delta}{1+\delta_2}x\right) + \frac{5}{\delta_2^2}\ex\left[ \frac{\max_{1 \leq i \leq 2^L}e^{-Q_i}}{\sum_{i=1}^{2^L} e^{-Q_i}} \right]  
    \end{align*}
for any $\delta,\delta_2, \epsilon > 0$ and any $L \geq 0$. As $\delta$ is arbitrarily small, continuity of measure implies that the inequality can be strengthened to
    \begin{eqnarray*}
\lefteqn{   \limsup_{n \to \infty} \prob\left(\frac{1}{e^{np}}|\mathcal{T}_{n,p}| > x\right) } \\
    & \leq& \prob\left(\sum_{i=1}^{2^L}2e^{-Q^-_{i,\epsilon}} > (1+\delta_2)^{-1}x\right) + \frac{5}{\delta_2^2}\ex\left[ \frac{\max_{1 \leq i \leq 2^L}e^{-Q_i}}{\sum_{i=1}^{2^L} e^{-Q_i}} \right].  
    \end{eqnarray*}
Then, Lemma \ref{finalcomputation} asserts that, for any $\delta_2,\delta_3 > 0$ we can choose $L$ large and $\epsilon$ small enough such that
    $$
    \prob\left(\sum_{i=1}^{2^L}2e^{-Q^-_{i,\epsilon}} > (1+\delta_2)^{-1}x\right) \leq \prob\left(E > (1+\delta_2)^{-1}x\right) + \delta_3
    $$
    and
    $$
    \frac{5}{\delta_2^2}\ex\left[ \frac{\max_{1 \leq i \leq 2^L}e^{-Q_i}}{\sum_{i=1}^{2^L} e^{-Q_i}} \right] \leq \frac{5}{\delta_2^2}\delta_2^3.
    $$
    Combining these with the above we get that
    $$
    \limsup_{n \to \infty} \prob\left(\frac{1}{e^{np}}|\mathcal{T}_{n,p}| > x\right) \leq \prob\left(E > (1+\delta_2)^{-1}x\right) + \delta_3 + 5\delta_2,
    $$
where $\delta_2,\delta_3$ are arbitrary. Taking $\delta_2,\delta_3 \downarrow 0$ and using the continuity of measure gives that
    $$
    \limsup_{n \to \infty} \prob\left(\frac{1}{e^{np}}|\mathcal{T}_{n,p}| > x\right) \leq \prob(E > x).
    $$
Repeating an almost identical procedure we obtain a matching lower bound,
    $$
    \liminf_{n \to \infty} \prob\left(\frac{1}{e^{np}}|\mathcal{T}_{n,p}| > x\right) \geq \prob(E > x).
    $$
This completes the proof of the first statement in Theorem \ref{sizethm}. 

Let $v_1,\ldots,v_m$ be the rank $1,\ldots,m$ children of the root in  $\mathcal{T}_{n,p}$. By the uniform spacings coupling there exist independent $\exponential(1)$ random variables $F_{1},\ldots,F_m$ such that
    $$
    \left(\exp(n\ell_{v_1}-np) , \ldots , \exp(n\ell_{v_m}-np)\right) \convdist (\exp(-F_{1}) , \ldots , \exp(-F_{1} - \ldots -F_{m}))
    $$
as $n \to \infty$. Conditioned on the labels of $v_1,\ldots,v_m$, the trees $\mathcal{T}_{n,p}(v_1),\ldots,\mathcal{T}_{n,p}(v_m)$ are all independent and distributed like $\mathcal{T}_{n,\ell(v_1)},\ldots,\mathcal{T}_{n,\ell(v_m)}$ trees. Thus, by the first part of this theorem,
    \begin{align*}
    \left(\frac{\mathcal{T}_{n,p}(v_1)}{e^{n\ell_{v_1}}},\ldots,\frac{\mathcal{T}_{n,p}(v_m)}{e^{n\ell_{v_m}}}\right) \convdist (E_1,\ldots,E_m),
    \end{align*}
where the sequence $(E_i)_{i \geq 1}$ is a sequence of independent $\exponential(1)$ random variables. Combining the two convergences with Slutsky's theorem implies 
    $$
    \left( \frac{\mathcal{T}_{n,p}(v_1)}{e^{np}} , \ldots , \frac{\mathcal{T}_{n,p}(v_m)}{e^{np}} \right) \convdist \left( E_1\exp(-E_{v_1}) , \ldots , E_m\exp(-E_{v_1} - \ldots - E_{v_m}) \right)
    $$
as $n \to \infty$. Since $e^{-E} \dist \unif[0,1]$ we are done.
\end{proof}

\section{The height}
\label{sec:height}

In this section, we prove Theorem \ref{heightthm}. We start with a technical lemma that follows directly from Cram\'er's large deviations theorem \cite{cramer1938, cramer1944, cramer+touchette2022}.

\begin{lemma}\label{cramer}
Let $k \geq 0$ be a fixed integer, and let $K$ be a uniform random variable on $\{ 1,\ldots,k\}$. For $i\in \{1,\ldots,k\}$, let $G_i$ be independent gamma$(i)$ random variables and let $X_1,\ldots,X_n$ be independently distributed as $G_K$. Then, there is a sequence $\phi(k)$ with $\phi(k) \to 0$ as $k \to \infty$, such that for any $0 < x < \frac{k+1}{2}$,
\begin{equation}\label{cramerbound}
\prob \left( \sum_{i=1}^n X_i \leq nx \right) \geq \exp\left(-n\left( \log \left( \frac{k}{ex}\right) + \phi(k) + o_n(1)\right)\right).
\end{equation}
\end{lemma}

\begin{proof}
By Cram\'{e}r's theorem \cite[Theorem 23.11]{klenke2008}, we have that, for $0 < x < \frac{k+1}{2}$,
$$
\liminf_{n \to \infty} \frac{1}{n}\log\prob\left( \sum_{i=1}^n X_i \leq nx \right) \geq - \inf_{0 < y < x} I(y) = -I(x)
$$
where
$$
I(x) = \sup_{\lambda \in \R} \left(\lambda x - \log\ex e^{\lambda X_1}\right)
$$
is the Legendre-Fenchel transform of the cumulant-generating function of $X_1$ (i.e., the logarithm of the moment generating function $\ex e^{\lambda X_1}$).
Observe that, for $\lambda \in (0, 1)$,
\begin{align*}
\ex \left( e^{\lambda X_1} \right) 
&= \frac{1}{k} \sum_{i=1}^k \ex \left( e^{\lambda G_i} \right)           \\
&= \frac{1}{k} \sum_{i=1}^k \frac{1}{(1 - \lambda)^i} \\
&= \frac{-1}{k\lambda}\left(1 - \frac{1}{(1-\lambda)^{k}}\right),
\end{align*}
and for $\lambda \geq 1$, $\ex e^{\lambda X_1} = \infty$. Altogether,
$$
\log \ex e^{\lambda X_1} = 
\left\{
\begin{array}{cc}
\infty, & \lambda \geq 1 \\
1, & \lambda = 0 \\
\frac{-1}{k\lambda}\left(1-(1-\lambda)^{-k}\right), & \text{otherwise}
\end{array}
\right.
$$
It is known that the function $J(\lambda) := \lambda x - \log \ex e^{\lambda X_1}$ is concave and continuous, and one can compute that $J(0) = 0$ and $J'(0) = x - \frac{k+1}{2}$. Together these facts imply that $J(\lambda) \leq 0$ for all $\lambda \geq 0$ when $0 < x < \frac{k+1}{2}$. Moreover, for $x > 0$, we have that $\lim_{\lambda \to -\infty} J(\lambda) = -\infty$. Noting that $J(-1/x) > 0$, the above facts imply that $J(\lambda)$ attains a global maximum at some $\lambda_* \in (\infty,0)$ when $0 < x < \frac{k+1}{2}$. Thus,
\begin{align*}
I(x) &= \sup_{\lambda < 0} \left( \lambda x + \log(k) + \log(-\lambda) - \log\left( 1 - \frac{1}{(1-\lambda)^k} \right) \right) \\
&\leq \log(k) + \sup_{\lambda > 0} (\lambda x + \log(-\lambda)) + \underbrace{\log\left(1-\frac{1}{(1-\lambda_*)^k}\right)}_{:= \phi(k)} \\
&= \log(k) - \log(x) - 1 + \phi(k)
\end{align*}
for all fixed $0 < x < \frac{k+1}{2}$. Note that $\phi(k) \to 0$ as $k \to \infty$. From here, recalling (\ref{cramerbound}) completes the proof.
\end{proof}
\medskip

Now we are ready to prove Theorem \ref{heightthm}.
For simplicity, we only present the proof for $p=1$. The extension to the general case is immediate.

\begin{proof}[Proof of Theorem \ref{heightthm}]
Recall that a vertex $u \in T_n$ is in the temporal tree $\mathcal{T}_n$ when the sequence of labels from the root to $u$ is monotonically decreasing. Thus, for a vertex in generation $d$ of $T_n$, 
$$
\prob ( u \in \mathcal{T}_n )  = \frac{1}{d!}.
$$
By the union bound, and Stirling's formula,
$$
\prob ( H_n \ge en ) 
\le \frac{n^{en}} {\lceil {en} \rceil !}
\le \frac{n^{en}} {\Gamma (en+1) }
\le \frac{1} {\sqrt{2 \pi en} }
= o_n(1).
$$
From here, it suffices to show that for every constant $\gamma < e$, and any integer $M > 0$,
$$\prob ( H_n \ge \gamma n  + M ) \to 1$$
as $n \to \infty$. 
We do this by exhibiting the existence of a vertex in $\mathcal{T}_n$ of depth $\ge \gamma n + M$, following a proof method that goes back to \cite{biggins1976,biggins1977}. For the rest of the proof, we view our tree as being constructed from the uniform spacings coupling of $T_n$.

Let $K>1$ be an integer.
We trim the entire tree $T_n$ by, for each $v \in T_n$, keeping only the $K$ children with the largest labels ordered from greatest to least label as $v_1,\ldots,v_K$. The result of this process is an infinite $K$-ary tree that we denote by $\mathcal{T}_n^{(K)}$. Recall that, using the uniform spacings coupling we may assume that, for $v \in \mathcal{T}_n^{(K)}$, 
$$
\ell_v = 1 - \sum_{u \in P^-(v)} S_{p(u),r(u)}^*,
$$
where $S_{v,i}^* = S_{v,1} + \cdots + S_{v,i}$ and $r(u)$ is the rank of vertex $u$. This means that the labels of vertices in $\mathcal{T}_n^{(K)}$ follow a generalized branching random walk in which the branching factor is $K$ at each generation and the step sizes are distributed like $(S_{1}^* , \ldots , S_{K}^*)$ where $S_i^* = S_1 + \cdots + S_i$ and $(S_1,\ldots,S_{n+1})$ is a vector of uniform spacings (for the rest of the proof we shall refer to these types of generalized step size branching random walks as just branching random walks). Altogether, this implies that
$$
\prob ( H_n \le \gamma n + M )
\le
\prob \left( \bigcap_{v \in \mathcal{T}_n^{(K)} : |v| = \lceil \gamma n + M \rceil} \left\{ \sum_{u \in P^-(v)} S_{p(u),r(u)}^* \ge 1 \right\} \right).
$$
Thus, what matters is the largest label of any vertex at depth $\lceil \gamma n +M \rceil$ in
a $K$-ary branching random walk in which the children have displacements
distributed as $S_1^*, \ldots, S_K^*$. Let $D$ be the maximal label of any vertex at distance $M$ from the root.
The $K^M$ vertices at distance $M$ from the root
have subtrees that behave in an i.i.d.\ manner,
and each vertex in these subtrees has a label at most equal to $D$ plus the
total displacement within its subtree.  
Therefore,
\begin{align}\label{basic}
\prob &\left( \bigcap_{v \in \mathcal{T}_n^{(K)} : |v| = \lceil \gamma n + M \rceil} \left\{ \sum_{u \in P^-(v)} S_{p(u),r(u)}^* \ge 1 \right\} \right) \notag \\
&\le
\prob \left( \bigcap_{v \in \mathcal{T}_n^{(K)} : |v| = \lceil \gamma n + M \rceil} \left\{ \sum_{u \in P^-(v)} S_{p(u),r(u)}^* \ge 1-D \right\} \right)^{K^M} \notag \\
&\le
\prob ( D > \epsilon ) + A_n(\epsilon)^{K^M},
\end{align}
where
$$
A_n(\epsilon) = \prob \left( \bigcap_{v \in \mathcal{T}_n^{(K)} : |v| = \lceil \gamma n + M \rceil} \left\{ \sum_{u \in P^-(v)} S_{p(u),r(u)}^* \ge 1-\epsilon \right\} \right).
$$
Recall that, jointly over $1 \le i \le K$,
$$
S_i^* \dist \frac { E_1 + \cdots + E_i } { E_1 + \cdots + E_{n+1} },
$$
where $E_1,\ldots,E_{n+1}$ are i.i.d.\ exponential random variables.
As $D$ is smaller than the sum of $2K^M$ random variables distributed as $S_K^*$,
we have by Markov's inequality,
$$
\prob ( D > \epsilon ) \le \frac{2K^M}{\epsilon} \frac{K}{n+1} = o_n(1).
$$
We show that for special choices of $\epsilon > 0, K > 0$,
\begin{align}\label{A_n}
 A_n(\epsilon) \le q < 1
\end{align}
for all $n$ large enough. Then the right-hand-side in (\ref{basic}) is upper bounded by $ o_n(1) + q^{K^M},$ which can be made as small as desired by taking $M$ large enough and letting $n$ tend to infinity. We conclude by establishing (\ref{A_n}). 

The dependence of the distribution of $S_i^*$ on $n$ is a slight inconvenience, so we consider a branching random walk with larger displacements. To that end, we introduce a $\ber(p_n)$ random variable $B_n$, where
$$
p_n = \prob \left( E_{K+1} + \cdots + E_{n+1} \le n (1-\epsilon) \right).
$$
Then, the values of vertices in the branching random walk on $\mathcal{T}_n^{(K)}$ defined by step sizes
\begin{align*}
 (1-B_n) \frac { E_1 + \cdots + E_i } { n(1-\epsilon) }  + B_n \quad 1 \leq i \leq K
\end{align*}
dominate the values for vertices in the original branching random walk described above. By the law of large numbers, for fixed $K$ and $\epsilon$,  $p_n \to 0$ as $n \to \infty$.
So, given $\delta > 0$, $p_n \le \delta$ for all $n$ large enough. Let $B$ be a $\ber(\delta)$ random variable. We introduce a final branching random walk where the family of step sizes, 
for $1 \leq i \leq K$, is
\begin{align*}
 W_i = 
 \left\{ \begin{array}{ll} 
 E_1 + \cdots + E_i  & \text{if $B=0$} \\
 \infty &
 \text{if $B=1$.} 
\end{array}
 \right.
 \end{align*} 
If we let, for all $1 \leq i \leq K$, and  $v \in \mathcal{T}_n^{(K)}$,
\begin{align*}
 W_{v,i} = 
 \left\{ \begin{array}{ll} 
 E_1 + \cdots + E_i  & \text{if $B_v=0$} \\
 \infty &
 \text{if $B_v=1$.} 
\end{array}
 \right.
 \end{align*} 
for $(E_{v,i} : 1 \leq i \leq K, v \in \mathcal{T}_n^{(K)})$ i.i.d. $\exponential(1)$ random variables and $B_v$ independent $\ber(\delta)$ random variables, we get that, for sufficiently large $n$,
\begin{equation}\label{anotherbound}
A_n(\epsilon)
\le
\prob \left( \bigcap_{v \in \mathcal{T}_n^{(K)} : |v| = \lceil \gamma n + M \rceil} \left\{ \sum_{u \in P^-(v)} W_{p(u),r(u)} \ge n(1-\epsilon)^2 \right\} \right).
\end{equation}
As the step sizes of this new branching random walk do not depend upon $n$,
we can identify a supercritical Galton-Watson process with an extinction probability that upper bounds the right-hand side of (\ref{anotherbound}).

Fix an integer $L$ and consider all $K^L$ vertices in generation $L$ of the new branching random walk with step sizes $(W_1,\ldots,W_K)$. If any of the Bernoulli random variables for the vertices in generation $\ell \le L$ from the root is one, then the root has no children. Otherwise, we set vertices in generation $L$ from the root to be a child of the root (in the new Galton-Watson process) if its label is $\le L(1-2\epsilon)/\gamma$. For each child of the root with the new Galton-Watson process, we repeat this procedure. That is, for each vertex $u$ with value $U$ in the branching random walk with step sizes $(W_1,\ldots,W_K)$ that is also in the Galton-Watson process, we add all vertices that are $L$ generations below $u$ that have values below $U + L(1-2\epsilon)/\gamma$ if none of the descendants of $u$ that are up to $L$ generations below $u$ have their Bernoulli value set to one. We call the resultant of this Galton-Watson process $G$.

If $|G| = \infty$, then the values of the vertices at level $j$ are $\le Lj (1-2\epsilon)/\gamma$ for all $j \geq 1$. These vertices correspond to vertices in the original tree at level $Lj$. In particular, conditioned on survival, there are vertices at level $\lceil \gamma n \rceil$ that have value
$$
\le \left\lfloor \frac {\lceil \gamma n \rceil}{L} \right\rfloor \times \frac{L(1-2\epsilon)}{\gamma} 
\le \frac{1 } {\gamma} + (1-2\epsilon) n 
< (1-\epsilon)^2 n
$$
for $n$ large enough. In particular, combined with (\ref{anotherbound}), this implies that $A_n(\epsilon) \leq \prob(|G| < \infty)$.

Define $c = (1-2\epsilon)/\gamma$. To determine the survival probability of $G$, we check that the expected number of children of the root, call it $G_1$, is larger than one. Indeed, 
\begin{align*}
\ex[G_1] &= (1-\delta)^{K^L}\sum_{|v| = L} \prob\left(\text{the value of $v$ is at most $Lc$}\right) \\
&= (1-\delta)^{K^L} K^L\prob\left(\text{a uniform vertex in generation $L$ has value at most $Lc$}\right) \\
&= (1-\delta)^{K^L} K^L \prob \left( G_{Y_1} + \cdots + G_{Y_L} \le  L c \right),    
\end{align*}
where $Y_1,\ldots,Y_L$ are i.i.d.\ random integers uniformly distributed
on $\{ 1, \ldots, K\}$, and $G_m$ stands for a gamma random
variable with parameter $m$. By Lemma \ref{cramer},
\begin{align*}
\prob \left( G_{Y_1} + \cdots + G_{Y_L} \le  L c \right) \geq \exp \left( -L \left( \log \left( \frac{ K }{ce } \right) + \phi(K)  + o_L(1) \right) \right) = \frac{(ce)^L}{K^L} e^{-L\phi(K) - o_L(L)}
\end{align*}
for $c \leq \frac{K+1}{2}$, and so
$$
\ex[G_1] \geq (1-\delta)^{K^L} (ce)^L e^{-L\phi(K) - o_L(L)}.
$$
For fixed $\gamma < e$, we can find $\epsilon > 0$ small enough such that $ce > 1$. Then, we can choose $L$, $K$ large enough and $\delta > 0$ small enough that the expected number of children is strictly larger than one. Then, with these chosen $c,L,K,\delta$, $G$ becomes extinct with some probability $q < 1$. As noted above, this implies that $A_n(\epsilon) \le q$ for all $n$ large enough, finishing the proof.
\end{proof}
\smallskip

\section{Typical depths}
\label{sec:depth}

In this section, we prove Theorem \ref{depththm}. As in the previous section, we present the proof for the $p=1$ case only, as the extension to the general case is straightforward.

Let $(\mathcal{Z}_k)_{k \geq 0}$ be the number of vertices in generation $k$ of the tree $\mathcal{T}_n$.  Since we argued previously that the height of $\mathcal{T}_{n} \convprob en$, it is sufficient to argue that $\mathcal{Z}_1 + \ldots + \mathcal{Z}_{(1-\epsilon)n}$ and $\mathcal{Z}_{(1+\epsilon)n} + \ldots + \mathcal{Z}_{2en}$ are both negligible compared to $|\mathcal{T}_n|$ for $\epsilon \in (0,1)$. One may do this using Markov's inequality, the union bound, and Stirling's formula. Indeed, for $\delta < \epsilon^2$,
\begin{align*}
    \prob\left( \mathcal{Z}_1 + \cdots + \mathcal{Z}_{(1-\epsilon)n} > e^{(1-\delta)n}\right) &\leq (1-\epsilon)n\prob\left( \mathcal{Z}_{(1-\epsilon)n} > \frac{1}{(1-\epsilon)n}e^{(1-\delta)n}\right) \\
    &\leq \frac{(1-\epsilon)^2n^2n^{(1-\epsilon)n}}{((1-\epsilon)n)!e^{(1-\delta)n}} \\
    &\leq C\frac{n^{2}e^{(1-\epsilon)n-(1-\delta)n}}{(1-\epsilon)^{(1-\epsilon)n}} \\
    &\leq Cn^2\exp\left( \delta n - \epsilon n + (1-\epsilon)\log(1-\epsilon) n \right) \\
    &\leq Cn^2\exp\left( \delta n - \epsilon n + (1-\epsilon)\epsilon n \right) \to 0,
\end{align*}
for some constant $C > 0$. For the other side, we repeat the same computation, with the only changes that $(1-\epsilon)$ is replaced with $(1+\epsilon)$ and the sum contains $2en$ terms now. Indeed, for $\delta < \epsilon^3$,
\begin{align*}
    \prob\left( \mathcal{Z}_{(1+\epsilon)n} + \ldots + \mathcal{Z}_{2en} > e^{(1-\delta)n} \right) &\leq \frac{(2e)^2n^2n^{(1+\epsilon)n}}{((1+\epsilon)n)!e^{(1-\delta)n}} \\
    &\leq C'\frac{n^2e^{(1+\epsilon)n}}{(1+\epsilon)^{(1+\epsilon)n}e^{(1-\delta)n}} \\
    &\leq C'n^2\exp\left( \epsilon n - \delta n - (1+\epsilon)\log(1+\epsilon)n \right) \\
    &\leq C'n^2\exp\left( \epsilon n - \delta n - (1+\epsilon)(\epsilon-\epsilon^2)n \right) \to 0,
\end{align*}
where $C'$ is again a positive constant. Finally, one more application of Markov's inequality allows us to assert that $\prob(|\mathcal{T}_n| \leq e^{(1-\delta/2)n}) \to 0$ as $n \to \infty$. Choosing $\delta = \epsilon^4$ (recall that we assume $\epsilon < 1$) we are able to conclude that 
$$
    \frac{\mathcal{Z}_1 + \cdots + \mathcal{Z}_{(1-\epsilon)n} + \mathcal{Z}_{(1+\epsilon)n} + \cdots + \mathcal{Z}_{2en}}{|\mathcal{T}_n|} \convprob 0
$$
as $n \to \infty$ for any $\epsilon > 0$. By definition, this means that
$$
    \frac{\mathcal{Z}_{(1-\epsilon)n} + \cdots + \mathcal{Z}_{(1+\epsilon)n}}{|\mathcal{T}_n|} \convprob 1
$$
as $n \to \infty$ for any $\epsilon > 0$, which means that uniformly chosen vertices will be between depth $(1-\epsilon)n$ and $(1+\epsilon)n$ for any $\epsilon > 0$ with probability tending to 1 as $n \to \infty$ which is the desired result.
\qed

\section{The expected degree distribution}
\label{sec:degrees}

In this section we prove Theorem \ref{degreethm}. Once again, for the sake of clarity of the presentation, we only consider the case $p=1$. The extension to the general case is immediate.

Consider a vertex $u$ of depth $\ell$ in $T_n$. The vertex has degree at least $k$ in $\mathcal{T}_n$ if and only if the labels of the $\ell$ edges on the path from the root to $u$ are decreasing, moreover, at least $k$ of the $n$ labels of the out-edges of $u$ have labels less than the minimum edge label on the path from the root to $u$. The probability of this event equals
\[
\frac{1}{\ell!} \cdot \frac{n}{n+\ell}\cdot \frac{n-1}{n-1 +\ell}\cdots 
\frac{n-k+ 1}{n-k+ 1 +\ell}~.
\]
Thus, the expected number of vertices of outdegree at least $k$ is
\[
\EXP L_{n,\ge k}
= \sum_{\ell=1}^\infty
\frac{n^\ell}{\ell!} \prod_{i=0}^{k-1}\frac{1}{1+\frac{\ell}{n-i}}
= e^n \EXP \prod_{i=0}^{k-1}\frac{1}{1+\frac{X_n}{n-i}}~,
\]
where $X_n$ is a \poisson($n$) random variable. Since 
$X_n/(n-i)\to 1$ in probability 
for each fixed $i\in \{0,1,\ldots,k-1\}$,
we have that for every $k\ge 1$,
\[
\lim_{n\to \infty} \frac{\EXP L_{n,\ge k}}{e^n} = 2^{-k}~, 
\]
which implies Theorem \ref{degreethm} for $p=1$.

\section{Appendix}

\subsection*{Proof of Lemma \ref{tightness}}

Recall that 
$\ex|\mathcal{T}_{n,p}| = e^{np}$. The second moment
of $|\mathcal{T}_{n,p}| = \sum_{v\in T_n} \1_{v\in \mathcal{T}_{n,p}}$
is a sum, over pairs of vertices of $T_n$, the probability that both vertices exist in $\mathcal{T}_{n,p}$. 
We may split the sum based on where the pairs of paths stop overlapping,
    \begin{align*}
    &\ex|\mathcal{T}_{n,p}|^2 \\
    &\leq \ex|\mathcal{T}_{n,p}|
    + 2\sum_{k=0}^{\infty}\sum_{\ell = k}^{\infty}\sum_{m=0}^k\underbrace{n^{k-m}\binom{n}{2}n^{m-1}n^{\ell - k + m - 1}p^{\ell+m}}_{I}\underbrace{\frac{1}{m!}\frac{1}{(\ell - k + m)!}}_{II}\underbrace{\frac{(\ell-k+2m)!}{(\ell+m)!}}_{III}.
    \end{align*}
Term $I$ comes from choosing the pairs of paths and ensuring the paths have edge labels below $p$, term $II$ is the probability that the portion of both paths after the overlap is decreasing ($m$ and $\ell-k+m$ edges), and term $III$ is the probability that the $(k-m)$ edges in the overlap are such that both paths as a whole are decreasing. Rearranging this expression and swapping the order of summation gives
    \begin{align*}
    \ex|\mathcal{T}_{n,p}|^2 &\leq \ex|\mathcal{T}_{n,p}| + \sum_{k=0}^\infty\sum_{\ell = k}^\infty\sum_{m=0}^k \frac{(np)^{\ell+m}}{(\ell+m)!}\binom{\ell - k + 2m}{m} \\ 
    &\leq \ex|\mathcal{T}_{n,p}| + \sum_{\ell = 0}^\infty\sum_{m=0}^\ell \frac{(np)^{\ell+m}}{(\ell+m)!}\left(\sum_{k=m}^\ell \binom{\ell-k+2m}{m}\right).
    \end{align*}
Writing $\binom{\ell-k+2m}{m} := f(k)$, we can bound the innermost sum with a geometric series. Since,
    $$
    \frac{f(k+1)}{f(k)} = \frac{(\ell-k+m)!}{(\ell-k+m)!}\frac{(\ell-k+2m-1)!}{(\ell-k+m-1)!} = \frac{\ell-k+m}{\ell-k+2m} \leq \frac{\ell}{\ell+m}
    $$
for all $k \geq m$, we have that
    \begin{align*}
    \sum_{k=m}^\ell \leq f(m)\sum_{j=0}^\infty \left(\frac{\ell}{\ell+m}\right)^j = f(m)\frac{1}{1-\frac{\ell}{\ell+m}} = f(m)\left(1 + \frac{\ell}{m}\right).
    \end{align*}
Splitting off the $m=0$ term from the original expression and applying this bound gives
    \begin{align*}
    \ex|\mathcal{T}_{n,p}|^2 &\leq \underbrace{\ex|\mathcal{T}_{n,p}|}_{I} + \underbrace{\sum_{\ell=0}^\infty \frac{(np)^\ell}{\ell!}2^\ell}_{II} + \underbrace{\sum_{\ell=0}^\infty\sum_{m=1}^\ell \frac{(np)^{\ell+m}}{(\ell+m)!}\binom{\ell+m}{m}\left(1+\frac{\ell}{m}\right)}_{III}.
    \end{align*}
Clearly $I = e^{np}$, $II = e^{2np}$, and $III$ can be bounded as,
    \begin{align*}
    III &\leq 2\sum_{\ell=0}^\infty\sum_{m=1}^\ell \frac{(np)^\ell}{\ell!}\frac{(np)^m}{m!}\frac{\ell}{m} \\
    &\leq 2e^{2np}\sum_{\ell=0}^\infty\sum_{m=1}^\ell \prob(X_n = \ell)\prob(Y_n = m) \frac{\ell}{m}\\\
    &\leq 2e^n\ex\left[ \frac{X_n}{Y_n}\1_{\{ 1 \leq Y_n \leq X_n \}} \right] \\
    &\leq 4e^{2n}\ex\left[ \frac{X_n}{Y_n+1} \right],
    \end{align*}
where $X_n$ and $Y_n$ are independent $\poisson(np)$ random variables. Finally, the value of the expected ratio given above is known to be $(1-e^{-np})$ (see, e.g., \cite{Coath2013}), so we can put everything together to get that
    $$
    \ex|\mathcal{T}_{n,p}|^2 \leq e^{np} + e^{2np} + 4e^{2np}\left(1-e^{-np}\right) \leq 5e^{2np} = 5(\ex|\mathcal{T}_{n,p}|)^2.
    $$
The second statement of the lemma is a quick corollary of the first. Since each of the trees are independent of one another, an application of (i) yields
$$
    \ex\left( \sum_{i=1}^m |\mathcal{T}_{n,p_i}|\right)^2 = \sum_{i=1}^m \ex|\mathcal{T}_{n,p_i}|^2 \leq 5\sum_{i=1}^m e^{2np_i}.
$$
Applying Chebyshev's inequality and upper bounding $e^{np_i} \leq \max_{1 \leq I \leq m} e^{np_i}$ gives
$$
\prob\left( \left|\sum_{i=1}^m |\mathcal{T}_{n,p_i}| - \mu \right| > \epsilon \mu \right) \leq \frac{5}{\epsilon^2}\frac{\left(\max_{1 \leq i \leq m}e^{np_i}\right)\sum_{i=1}^m e^{np_i}}{\left(\sum_{i=1}^m e^{np_i}\right)^2}~.
$$    
Factoring an $e^{np}$ (recall that $q_i = p-p_i$) from both the numerator and denominator completes the proof:
$$
\prob\left( \left|\sum_{i=1}^m |\mathcal{T}_{n,p_i}| - \mu \right| > \epsilon \mu \right) \leq \frac{5}{\epsilon^2}\frac{\max_{1 \leq i \leq m}e^{np}e^{-nq_i}}{e^{np}\sum_{i=1}^m e^{-nq_i}} = \frac{5}{\epsilon^2}\frac{\max_{1 \leq i \leq m}e^{-nq_i}}{\sum_{i=1}^m e^{-nq_i}}~.
$$
\qed

\subsection*{Proof of Lemma \ref{treeforest}}

Let $v \in I(L+1)$, and let $v_1,\ldots,v_n$ be the children of $p(v)$ in $T_n$. Take some sequence $(E_i)_{i \geq 1}$ such that
    $$
    S_i := S_{p(v),i} = \frac{E_i}{E_1 + \cdots + E_{n+1}}.
    $$
for all $1 \leq i \leq n+1$ where the collection $(S_{v,i} : v \in T_n , 1 \leq i \leq n+1)$ are the spacings from the uniform spacings coupling of $T_n$. The only dependence of the collection $(S_i : r(v)+1 \leq i \leq n+1)$ upon the collection $(S_1 , \ldots , S_{r(v)})$ comes from the existence of the random variables $E_1,\ldots,E_{r(v)}$ in the denominator. We define new spacings for $r(v) + 1 \leq i \leq n+r(v)+1$ that are independent of this collection,
    $$
    S_i^* = \frac{E_i}{E_{r(v)+1} + \cdots + E_{n+r(v)+1}}.
    $$
Now we construct a coupling of $\mathcal{F}(v)$ and a tree $\mathcal{T}_{new}$ that is distributed like $\mathcal{T}_{n,\ell_v^-(\epsilon)}$ from these spacings. Let $u_1,\ldots,u_n$ be $n$ vertices attached to a mutual root. We give these vertices labels
    $$
    \ell_{u_i}^* := \ell_v^-(\epsilon) - S_{r(v) + 1}^* - \cdots - S_{r(v) + k}^*.
    $$
Then, we attach $T_n(v_{r(v)+i})$ to the vertices $u_i$ for, making the vertex $u_i$ the root for $1 \leq i \leq n - r(v)$. For the vertices $u_i$ with $n < i \leq n+r(v)$ we make $u_i$ the root of independently sampled trees distributed like $T_n$. We call this new tree $T_{new}$, and upon deleting non-decreasing paths and vertices with negative labels (the negative label vertices can only exist in the first generation), we obtain a tree $\mathcal{T}_{new}$. Note that, since we are dealing with vertices of fixed finite index, the vertices $u_1,\ldots,u_k$ are in $\mathcal{T}_{new}$ and $v_1,\ldots,v_k$ are in $\mathcal{T}_{n,p}$ with probability tending to 1 as $n \to \infty$ for $k = \lfloor n^{1/4} \rfloor$. This is because the degree of a vertex is distributed like a $\bin(n,\ell_v)$ random variable and $\ell_v$ is of constant order when $v$ has a fixed finite index. Moreover, by the results on the uniform spacings coupling from Section 2, indeed, $\mathcal{T}_{new}$ is distributed like $\mathcal{T}_{n,\ell_v^-(\epsilon)}$. Finally, we note that by construction $\mathcal{T}_{new}$ is independent of $T_n(v)$ and of $T_n(p(v))$ when we condition on the value $\ell_v^-(\epsilon)$.

To complete the proof, we show that $|\mathcal{T}_{n,p}(v_i)| \leq |\mathcal{T}_{new}(u_i)|$ for $1 \leq i \leq k$ with sufficiently high probability and that
    \begin{equation}\label{conv}
    \frac{|\mathcal{F}(v)|}{\sum_{i=1}^k \mathcal{T}_{n,p}(v_i)} \convprob 1
    \ \text{ and } \
    \frac{|\mathcal{T}_{new}|}{\sum_{i=1}^k \mathcal{T}_{new}(u_i)} \convprob 1
    \end{equation}
as $n \to \infty$.

Since $\ell_v^-(\epsilon) \geq \ell_v$ it holds that for $1 \leq i \leq k$ that
    $$
    \ell_{v_i}-\ell_{v_i}^* \leq \left|\sum_{j=1}^i (S_{r(v)+j}-S_{r(v)+j}^*)\right|,
    $$
with probability tending to 1 as $n \to \infty$. Writing $R = \sum_{j=1}^{r(v)} E_{n+1+j} - \sum_{j=1}^{r(v)} E_j$ and $T = \sum_{i=1}^{n+1} E_i$ we can simplify the above to get
    $$
    \ell_{v_i}-\ell_{v_i}^* \leq \frac{\left(1+\frac{R}{T}\right)\sum_{j=1}^i E_{r(v)+j}}{\left(1 + \frac{R}{T}\right)T},
    $$
and so
    $$
    \sup_{1 \leq i \leq k} (\ell_{v_i}-\ell_{v_i}^*) \leq \frac{\frac{R}{T}\sum_{j=1}^k E_{r(v)+j}}{\left(1 + \frac{R}{T}\right)T} \leq \frac{\left(1+\frac{R}{T}\right)\sum_{j=1}^k E_{r(v)+j}}{T} \leq \frac{\left(1+\frac{R}{T^*}\right)\sum_{j=1}^k E_{r(v)+j}}{T^*}
    $$
with probability tending to 1 as $n \to \infty$, where $T \geq \sum_{i = r(v)+k+1}^{n+1} E_{i} := T^*$. Note that all the factors in the final upper bound are independent. Applying the law of large numbers we get that, for any $\delta > 0$,
    \begin{equation}\label{supbound}
    \sup_{1 \leq i \leq k} (\ell_{v_i}-\ell_{v_i}^*) \leq \frac{(1+\delta)x_nk}{n^2}
    \end{equation}
with probability tending to 1 as $n \to \infty$, where $x_n$ is any sequence that is $\omega_n(1)$ (since $R$ is a finite sum it won't tend to infinity). By our construction, the subtree $\mathcal{T}_{n,p}(v_i)$ can only be larger than the subtree $\mathcal{T}_{new}(u_i)$ if at the root the first tree has a child that the second did not, i.e., one of the $n$ uniforms sampled was in the interval $[\ell_{v_i}^*,\ell_{v_i}]$. Hence, by conditioning on (\ref{supbound}) and applying the union bound we get, 
    \begin{align*}
    \prob\left( \bigcup_{i=1}^k \{|\mathcal{T}_{n,p}(v_i)| \geq |\mathcal{T}_{new}(u_i)|\} \right) &\leq k\prob\left(\bin\left(n,\frac{(1+\delta)x_nk}{n^2}\right) > 0\right) + o_n(1) \\ 
    &\leq \frac{(1+\delta)k^2x_n}{n} + o_n(1) = o_n(1)
    \end{align*}
when one chooses $x_n$ accordingly. All that is left is to show (\ref{conv}). The proofs of both convergences can be done by an almost identical method as a consequence of the uniform spacings coupling so we only present the proof for the case of $|\mathcal{T}_{new}|$.

Clearly, to prove the convergence it is enough to show that
    \begin{equation}\label{finalconv}
    \frac{\sum_{i=k+1}^n |\mathcal{T}_{new}(u_i)|}{|\mathcal{T}_{new}|} \convprob 0
    \end{equation}
as $n \to \infty$, where $\mathcal{T}_{new}(v_i) = \emptyset$ if the vertex $v_i$ does not exist in the tree. Since $\ell_{v_i}^* \leq \ell_{v_k}^* := p_k$ (for this part of the proof define $p = \ell_v^-(\epsilon)$) for all $i \geq k$ we instead show a stronger result. We identify sequences $(a_n)_{n \geq 0}$, $(b_n)_{n \geq 0}$ with $a_n = o_n(b_n)$ such that, for i.i.d.\ $\mathcal{T}_{n,p_k}$ distributed trees, $\mathcal{T}_{n,p_k}^{(1)},\ldots,\mathcal{T}_{n,p_k}^{(n)}$,
    $$
    \prob\left(\sum_{i=1}^n |\mathcal{T}_{n,p_k}^{(i)}| > a_n \right) \convprob 0
    $$ 
as $n \to \infty$. We also observe that $\prob\left(|\mathcal{T}_{new}| \leq b_n \right) \convprob 0$ as $n \to \infty$. Combining the two results with the fact that the numerator in (\ref{finalconv}) is upper bounded by the first sum allows us to conclude (\ref{finalconv}). First, from a simple law of large numbers argument similar to the one above it holds that for any $0 < \delta < 1$,
    \begin{equation}\label{ineq}
    p_k \leq p - \frac{1-\delta}{n^{3/4}}
    \end{equation}
with probability tending to 1 as $n \to \infty$. Let $\gamma_n>0$ be some sequence to be specified later. Conditioning on (\ref{ineq}) and applying Markov's inequality, we have
    \begin{align*}
    \prob&\left( \sum_{i=1}^{n} |\mathcal{T}_{n,p_k}^{(i)}| > \exp(np - \gamma_n)\right) \\
    &\leq n\ex\left[\prob\left( |\mathcal{T}_{n,p-(1-\delta)/n^{3/4}}| > \frac{1}{n}\exp(np - \gamma_n) \Big| p \right)\right] + o_n(1) \\
    &\leq n\ex\left[ \frac{n\exp(np - (1-\delta)n^{1/4})}{\exp(np - \gamma_n)} \right] + o_n(1) \\
    &\leq n^2\exp(\gamma_n - (1-\delta)n^{1/4})~,
    \end{align*}
which converges to $0$ as $n \to \infty$ whenever $\gamma_n = o_n(n^{1/4})$. For $m = \lfloor n^{1/8} \rfloor$, $p_m = \ell_{v_m}^*$ one can show that
    $$
    p_m \geq p - \frac{1+\delta}{n^{7/8}}
    $$
with probability tending to 1 as $n \to \infty$. Since the labels of the vertices $v_1,\ldots,v_m$ are all larger than that of $v_m$ it holds that, from Lemma \ref{tightness} and the Chebyshev-Cantelli inequality,
    \begin{align*}
    \prob&\left( |\mathcal{T}_{new}| \leq \exp(np - \beta_n) \right) \\
    &\leq \prob\left( |\mathcal{T}_{n,p-(1+\delta)/n^{7/8}}| \leq \exp(np - \beta_n) \right)^m \\
    &\leq \ex\left[ \left( \frac{\var|\mathcal{T}_{n,p-(1+\delta)/n^{7/8}}|}{\var|\mathcal{T}_{n,p-(1+\delta)/n^{7/8}}| + \frac{1}{5}\var|\mathcal{T}_{n,p-(1+\delta)/n^{7/8}}|(1-e^{(1+\delta)n^{1/8}-\beta_n})^2} \right)^m \right],
    \end{align*}    
which also converges to 0 as $n \to \infty$ whenever $\beta_n = \omega_n(n^{1/8})$. Altogether, choosing a $\beta_n = n^{1/8+\alpha}$ and $\gamma_n = n^{1/4-\alpha}$ for sufficiently small $\alpha > 0$ is enough to complete the proof.
\qed

\subsection*{Proof of Lemma \ref{finalcomputation}}

Recall that by definition, $X_L = \sum_{|v| = L} e^{-nY_v}$, where the vertices are in $\mathcal{T}^*$. Each vertex $v$ has two children, say $v_1$ and $v_2$. These children have values $Y_{v_1} =Y_v + \frac{1}{n}E(1,v)$ and $Y_{v_2} \dist Y_v + \frac{1}{n}E(2,v)$, where $E(i,v) \dist \frac{E}{n}$ for any pair $(i,v)$ and are independent of all other edge labels in the graph. Since $nE(i,v) \dist E$ for all $(i,v)$, it holds that $U(i,v):= e^{-nE(i,v)} \dist \unif[0,1]$. Thus,
\begin{align*}
    \ex\left[ X_{L+1} | X_L \right] 
    &=\sum_{|v| = L}\sum_{i=1}^2 \ex\left[U(i,v)e^{-nY_v} | X_L\right] \\
    &= \sum_{|v| = L}\sum_{i=1}^2 \frac{1}{2}\ex[e^{nY_v}|X_L] \\
    &= \ex\left[ \sum_{|v| = L} e^{-nY_v} \Big| X_L \right] = X_L~.
\end{align*}
Hence, $X_L$ is a martingale with respect to $L$ with $\sup_L\ex[X_L] < \infty$ for any $n$, and so it has an almost sure limit. Call this limit $X$. 

Due to the structure of $\mathcal{T}^*$, going down one step in the tree reveals two copies of $\mathcal{T}^*$, both of which have an extra exponential from the first edge in all the vertex values. This structural recursion for the tree implies a distributional equality for the branching random walk:
$$X_L \dist U(X_{L-1} + X_{L-1}'),$$
where $U \dist \unif[0,1]$ and $X_{L-1}$ and $X_{L-1}'$ are two independent copies of $X_{L-1}$. From this, we obtain the distributional identity, $X \dist U(X' + X'')$ for $X'$ and $X''$ independent of each other.

Set $a_k := \ex[X^k]$. Using the distributional identity for $X$, we obtain the recursion $a_0 = 1$, $a_1 = \frac{1}{2}$,
$$
    a_k = \frac{1}{k+1}\sum_{i=0}^k\binom{k}{i}a_ia_{k-i}~.
$$
It is easily verified that $a_k = \frac{k!}{2^k}$ is a solution of this recursion. Thus, $\ex[X^k] = \frac{k!}{2^k}$ for all $k \geq 1$, which implies that $X \dist \frac{E}{2}$. The uniqueness of this distributional identity follows from noticing that the exponential distribution satisfies the Stieltjes moment problem conditions (\cite{Durrett2019}). This covers the first claim.

The second claim is a consequence of the Biggins-Hammersley-Kingman theorem (see, e.g., \cite{Addario2009}). For our purposes, the theorem implies that both the minimum and maximum value of all vertices in the $L$-th generation of a branching random walk with step size $(1\pm\epsilon)\frac{E}{n}$ is $\Theta(\frac{L}{n})$ as $L \to \infty$. More precisely, we have constants $C_1,C_2 > 0$ such that
    $$
    \prob\left( C_1L \leq \min_{1 \leq i \leq 2^L} nq_i \leq \max_{1 \leq i \leq 2^L} nq_i \leq C_2L \right) \to 1 \quad \text{as $L \to \infty$}.
    $$
If we take $\epsilon(L) = L^{-2}$, then $\min_{1 \leq i \leq 2^L} n\epsilon(L) q_i$ and $\max_{1 \leq i \leq 2^L} n\epsilon(L)q_i$ both converge to 0 in probability as $L \to \infty$, and so $\min_{1 \leq i \leq 2^L}\exp(\pm n\epsilon(L)q_i)$ and $\max_{1 \leq i \leq 2^L} \exp(\pm n\epsilon(L)q_i)$ both converge to 1 in probability as $L \to \infty$. Now, from the definition of the values $(q_i^\pm(\epsilon))_{i=1}^{2^L}$, we obtain the bounds
    \begin{align*}
    \min_{1 \leq i \leq 2^L} \exp(\mp n\epsilon q_i)X_L \leq X_{L}^\pm \leq \max_{1 \leq i \leq 2^L} \exp(\mp n\epsilon q_i)X_L~.  
    \end{align*}
From Slutsky's theorem and the first claim (i) it holds that both the upper and lower bounds above converge to $\frac{E}{2}$ as $L \to \infty$, and so the same holds for $X_L^{\pm}$.

For the final claim,  note that the aforementioned Biggins-Hammersley-Kingman theorem states that $\max_{1 \leq i \leq 2^L} \exp(-n q_i^\pm(\epsilon)) \to 0$ almost surely as $L \to \infty$. Since 
$$
\max_{1 \leq i \leq 2^L} e^{- n\epsilon q_i^\pm(\epsilon)} \leq 1,
$$
the convergence also holds in $L_1$. Now, let $\eta > 0$ and let $\epsilon(L)$ be as in the second claim. Splitting up the expectation in (iii) gives the upper bound
\begin{eqnarray*}
\lefteqn{    \ex\left[ \frac{\max_{1 \leq i \leq 2^L} \exp(-nq_i^\pm(\epsilon))}{\sum_{i=1}^{2^L} \exp(-nq_i^\pm(\epsilon))} \right] } \\
    &=& \ex\left[ \frac{\max_{1 \leq i \leq 2^L} \exp(-nq_i^\pm(\epsilon))}{\sum_{i=1}^{2^L} \exp(-nq_i^\pm(\epsilon))}\1_{\{\sum_{i=1}^{2^L} \exp(-nq_i^\pm(\epsilon)) < \eta\}} \right] \\
    &  & + \ex\left[ \frac{\max_{1 \leq i \leq 2^L} \exp(-nq_i)}{\sum_{i=1}^{2^L} \exp(-nq_i^\pm(\epsilon))}\1_{\{\sum_{i=1}^{2^L} \exp(-nq_i^\pm(\epsilon)) > \eta\}} \right] \\
    & \leq & \prob\left( \sum_{i=1}^{2^L} \exp(-nq_i^\pm(\epsilon)) < \eta \right) + \frac{1}{\eta}\ex\left[\max_{1 \leq i \leq 2^L}\exp(-nq_i^\pm(\epsilon)) \right].
\end{eqnarray*}
As $L \to \infty$ the final upper bound converges to $\prob\left(E < 2\eta\right)$. From here letting $\eta \downarrow 0$ completes the proof. To prove that 
    $$
    \ex\left[ \frac{\max_{1 \leq i \leq 2^L} \exp(-nq_i)}{\sum_{i=1}^{2^L} \exp(-nq_i)} \right] \to 0
    $$
as $L \to \infty$, it suffices to follow the same procedure as the $\pm$ case just covered above.
\qed

\setstretch{1.0}

\end{document}